\documentclass[english]{article}

\usepackage{hyperref}
\usepackage[T1]{fontenc}
\usepackage{lmodern}
\usepackage[left,modulo]{lineno}
\usepackage{float}
\usepackage{babel}
\usepackage{geometry}
\usepackage{mathtools}
\usepackage{amsmath}
\usepackage{amsthm}
\usepackage{amssymb}
\usepackage{graphicx}
\usepackage{setspace}
\usepackage{xcolor}
\usepackage[square,numbers]{natbib}
\usepackage{enumitem}

\makeatletter

\theoremstyle{plain}
\newtheorem{thm}{\protect\theoremname}
\theoremstyle{remark}
\newtheorem{rem}[thm]{\protect\remarkname}

\linenumbers

\mathtoolsset{showonlyrefs,showmanualtags}	

\renewcommand{\o}{\mathrm{ol}}
\newcommand\Exp{\mathbb{E}}

\newcommand\R{\mathbb{R}}

\renewcommand{\L}{K} 
\newcommand{\K}{L} 
\newcommand{\LL}{\L^{\o}}
\newcommand{\KK}{L^\o} 
\newcommand{\E}{S} 
\renewcommand{\S}{P} 
\renewcommand{\P}{\varSigma} 
\newcommand{\drift}{f}
\newcommand{\diff}{G}
\newcommand{\odrift}{g}
\newcommand{\odiff}{H}
\newcommand{\s}{s} 
\newcommand{\x}{\delta x} 
\newcommand{\uu}{\delta u} 
\newcommand{\y}{\delta y} 
\newcommand{\hx}{\delta \hat{x}} 
\renewcommand{\lq}{\mathrm{lq}}
\newcommand{\JJ}{\mathcal{J}^*_{app}}

\renewcommand{\bf}[1]{ \boldsymbol{#1}}

\makeatother

\providecommand{\remarkname}{Remark}
\providecommand{\theoremname}{Theorem}

\begin{document}



\nolinenumbers


\title{ Stochastic Optimal Feedforward-Feedback Control \\ for Partially Observable Sensorimotor Systems}

\author
{Bastien Berret$^{1}$ and Fr\'{e}d\'{e}ric Jean$^{2}$\\
\\
\normalsize{$^1$Universit\'{e} Paris-Saclay, Inria, CIAMS, 91190 Gif-sur-Yvette, France.}\\
\normalsize{$^2$UMA, ENSTA, Institut Polytechnique de Paris, 91120 Palaiseau, France.}\\
\normalsize{E-mails: bastien.berret@universite-paris-saclay.fr; frederic.jean@ensta.fr}\\}

\date{}

\maketitle













\begin{abstract} 
Robust control of complex engineered and biological systems hinges on the integration of feedforward and feedback mechanisms. This is exemplified in neural motor control, where feedforward muscle co-contraction complements sensory-driven feedback corrections to ensure stable behaviors. However, deriving a general continuous-time framework to determine such optimal control policies for partially observable, stochastic, nonlinear, and high-dimensional systems remains a formidable computational challenge. Here, we introduce a framework that extends neighboring optimal control by enabling the feedforward plan to explicitly account for feedback uncertainties and latencies. Using statistical linearization, we transform the stochastic problem into an approximately equivalent deterministic optimization within a tractable, augmented state space that retains critical nonlinearities, offering both mechanistic interpretability and theoretical guarantees on approximation fidelity. We apply this framework to human neuromechanics, demonstrating that muscle co-contraction emerges as an optimal adaptation to task demands, given the characteristics of our sensorimotor system. Our results provide a computational foundation for neuromotor control and a generalizable tool for the control of nonlinear stochastic systems.
\end{abstract}


\section*{Introduction}

Achieving robust behavior in the face of delays and uncertainties is a fundamental challenge for both engineered and biological systems. In human sensorimotor control, systemic latencies (>50 ms) render pure feedback control insufficient for tasks such as drilling a hole \citep{Leib2023}. Consequently, the central nervous system (CNS) employs anticipatory muscle co-contraction \citep{Phan2020}---a ubiquitous strategy observed across tasks from tool use to locomotion \citep{Voloshina2013}. By exploiting the nonlinear, activation-dependent viscoelasticity of muscles \citep{Hogan1984a,Burdet2013}, the CNS regulates mechanical impedance \citep{Burdet2001,Franklin2011,Berret2024}, thereby mitigating the consequences of sensory noise and delays on state feedback control \citep{Franklin2007,Scott2012}.

While human sensorimotor control highlights the critical coupling of feedforward and feedback mechanisms, a theoretical framework to optimally combine them remains elusive. Specifically, current mathematical tools struggle to address the nonlinear, high-dimensional, and partially observable nature of systems subject to
sensorimotor latencies. Solving these stochastic optimal control problems in continuous time is notoriously difficult beyond the standard Linear-Quadratic-Gaussian (LQG) case, which is a dominant yet limiting modeling framework for the neural control of movement \citep{Todorov2002}. Classical formulations describe the optimal solution through infinite-dimensional, computationally intractable equations in belief space \citep{Zakai1969,FlemingPardoux1982,bensoussan1992stochastic}. This has motivated the development of approximate numerical methods, most notably trajectory optimization approaches like Differential Dynamic Programming (DDP) \citep{JacobsonMayne1970} and its stochastic variants such as iterative LQG (iLQG) \citep{Todorov2005b, Li2007} and stochastic DDP \citep{Theodorou2010a}. These algorithms iteratively approximate dynamics and quadratize costs to generate locally optimal feedback control policies, yet they often struggle to seamlessly integrate state estimation with planning under uncertainty. To address partial observability, belief-space extensions like B-DDP \citep{Platt2010, VandenBerg2012} incorporate uncertainty propagation via approximate filters, though often at the cost of high computational complexity. Other paradigms, such as Stochastic Model Predictive Control \citep{Mesbah2018} or sampling-based path-integral frameworks \citep{Todorov2009, Theodorou2010, Williams2017}, offer alternative ways to handle uncertainty. However, like modern deep reinforcement learning \citep{Okada2020, FanMing2021, Wang2023}, these methods prioritize empirical scalability over the interpretability and theoretical guarantees essential for modeling complex neuromechanical systems.

To address these limitations, we introduce a nonlinear stochastic optimal control framework for partially observable systems that leverages efficient numerical tools while preserving interpretability and approximation guarantees. Our approach falls within the category of trajectory optimization and, similar to iLQG and B-DDP, utilizes Kalman-type filters for uncertainty propagation. However, we depart from traditional neighboring optimal control---which designs feedback only after planning a nominal deterministic path \citep{Bryson1969, Athans1971}---by accounting for stochasticity \emph{during} the planning phase \citep{Berret2020b, Berret2021}. By restricting the control policy to an affine dependence on the state estimate, we merge feedforward and feedback components into a unified dynamic optimization problem. We demonstrate that through statistical linearization \citep{Berret2020, Leparoux2024}, we can reformulate this challenging stochastic problem as a deterministic optimal control problem in a tractable, augmented space, while preserving the critical nonlinear properties and covariance controllability essential for robust motor control. The resulting problem is solvable using state-of-the-art tools for deterministic trajectory optimization. Application of this framework to human sensorimotor control illuminates the fundamental interdependence between muscle co-contraction and state feedback control.

\section*{Results}

\subsection*{Partially observable stochastic optimal control problem}
\label{se:soc}

\paragraph*{Problem formulation and hypotheses}
Let $x_t \in \R^n$, the state of a system at time $t$,  be described by an It\^{o} process $x=\{ x_t \}$ solution of a stochastic control system
\begin{equation} \label{eq:sde}
    dx_t = \drift(t,x_t,u_t) dt + \diff(t,x_t,u_t) dw_t, \quad x_0=\xi,
\end{equation}
and the observation at time $t$, $y_t \in \R^p$, be described by an It\^{o} process $y=\{ y_t \}$ of the form
\begin{equation} \label{eq:sde_obs}
    dy_t = \odrift(t,x_t) dt + \odiff(t,x_t) dv_t, \quad y_0=0,
\end{equation}
where $w = \{ w_t \}$, $v = \{ v_t \}$ are two independent Wiener processes of dimensions $k \le n$ and $p$ respectively, and $\xi$ is a given random variable of dimension $n$ (independent of $w_t,v_t$). A control process $u=\{u_t\}$ (where $u_t \in \R^m$) is called \emph{admissible} if it is $\{\mathcal{Y}_t\}_{t\ge 0}$-adapted, where $\mathcal{Y}_t$ is the $\sigma$-algebra generated by $\{y_s,  0\le s \le t \}$. The  general stochastic optimal control problem under consideration is to find an admissible control which minimizes the cost functional
\begin{equation}
    \mathbb{J}(u) = \Exp [J(u)], \qquad \hbox{with} \quad  J(u)= \psi(x_T) + \int_0^{T} \ell(t,x_t,u_t) dt,
\end{equation}
for a fixed or free time horizon $T>0$, possibly subject to some additional constraints in expectation or in probability on the state.

Except for the LQG case, such a problem is extremely difficult to analyze and solve numerically. It is then natural to look for suboptimal solutions by choosing a restricted class of admissible controls for which the problem is easier to handle.
In this paper, we consider a context characterized by the following points:
\begin{itemize}
  \item Movement execution is preceded by a planning phase that allows some parameters of the control policy to be determined off-line, from prior experience and knowledge;
  \item Parameter uncertainties and random disturbances are rather small and approximately Gaussian (no multimodal distributions);
  \item Observation is partial and possibly of poor quality due to systemic noise and delays (see below).
\end{itemize}
As a consequence, feedforward (open-loop, deterministic) controls and average behaviors play a central role, and we can rely on Kalman-type estimators. Inspired by the concept of \emph{neighboring-optimal stochastic control} (see \cite{stengel1986optimal,Bryson1969,Athans1971,maybeck1982stochastic} with different terminologies),
we propose to consider admissible feedback controls of the form described below.

Fixing $T>0$ for simplicity (see Remark\,\ref{re:extensions} regarding free time), we denote $\mathcal{L}^2 := \mathcal{L}^2([0,T],\R^m)$ and $\mathcal{L}_M^2 := \mathcal{L}^2([0,T],M_{m,n}(\R))$. For an open-loop control $u^\o \in \mathcal{L}^2$ and an open-loop gain $\K \in \mathcal{L}_M^2$, we define the control process $u = \pi (u^\o, \K)$ by
\begin{equation}
u_t  = \pi (u^\o, \K)(t) := u^\o(t) + \K(t) \left( \hat{x}_t - m^\o(t)\right), \quad t \in [0,T],
\end{equation}
where:
\begin{itemize}
  \item $m^\o(\cdot)$ is the trajectory of the deterministic part of the system, i.e., the solution of the differential equation
  \begin{equation}\label{eq:moyenne}
  \dot m^\o (t) = \drift(t,m^\o(t),u^\o(t)), \quad m^\o(0)=\Exp [\xi];
  \end{equation}
  \item $\hat{x}$ is the linearized Kalman filter around $(m^\o(\cdot),u^\o(\cdot))$ (see formula~\eqref{eq:LKF} below).
\end{itemize}
It is worth noting that the equations defining $m^\o$ and $\hat{x}$ depend only on  $u^\o$ and $\K$, which justifies the notation $u = \pi (u^\o, \K)$. Moreover, the same affine structure for the control policy $\pi$ arises in related methods such as iLQG and DDP \citep{Li2007,Theodorou2010a}.
Finally, we define the restricted class $\mathcal{U}$ of admissible stochastic controls as
\begin{equation}
\mathcal{U} = \left\{ \pi (u^\o, \K) \ : \ u^\o  \in \mathcal{L}^2 \ \mathrm{and} \ \K \in \mathcal{L}_M^2 \right\}.
\end{equation}

Therefore, we tackle in this paper the following restricted stochastic optimal control problem:
\begin{equation}
  \min_{u \in \mathcal{U}} \mathbb{J}(u) = \min_{ \begin{subarray}{l} u^\o \in \mathcal{L}^2 \\ \K \in \mathcal{L}_M^2 \end{subarray} } \mathcal{J}(u^\o, \K), \qquad \hbox{where } \mathcal{J}(u^\o,\K)= \mathbb{J} \left( \pi (u^\o, \K) \right).
\end{equation}
The left-hand expression is a stochastic optimal control problem while the right-hand one is a dynamic optimization problem in $\mathcal{L}^2$ spaces, which may still be difficult to solve, however. Let us show that the latter can be approximated by a deterministic optimal control problem in finite dimension, for which powerful numerical and theoretical tools exist \citep{betts2020}. 

\begin{rem}
Let us recall the definition of the linearized Kalman filter. Given an open-loop control $u^\o (\cdot) \in \mathcal{L}^2$ and the associated solution $m^\o(\cdot)$ of~\eqref{eq:moyenne}, we define the matrices of the linearized system around $(m^\o(\cdot),u^\o(\cdot))$ in the following way:
\begin{align}
A(t) := \frac{\partial \drift}{\partial x}(t,m^\o(t),u^\o(t)), \quad B(t) &:= \frac{\partial \drift}{\partial u}(t,m^\o(t),u^\o(t)), \quad C(t) := \frac{\partial \odrift}{\partial x}(t,m^\o(t)), \\
 G(t):= \diff(t,m^\o(t),u^\o(t)), & \quad H(t) := \odiff(t,m^\o(t)).
\end{align}
Then the linearized Kalman filter around $(m^\o(\cdot),u^\o(\cdot))$ is the controlled stochastic process $\hat{x}=\{\hat{x}_t\}$ defined by
\begin{equation}\label{eq:LKF}
d \hat{x}_t = \drift(t,\hat{x}_t,u_t) dt + \L(t) (dy_t - \odrift(t,\hat{x}_t) dt), \quad \hat{x}_0=\Exp[\xi],
\end{equation}
where the matrix-valued function $\L(\cdot)=\LL(\cdot)$ (which implicitly depends on $u^\o$) is defined by
\begin{align}\label{eq:LKFgain}
  & \L(t) = \S(t) C(t)^\top (H(t)H(t)^\top)^{-1}  \qquad \hbox{where $\S(\cdot)$ is the solution of} \\
  & \dot \S = A \S + \S A^\top - \S C^\top (H H^\top)^{-1} C \S + G G^\top, \quad \S(0) = \mathrm{cov}(\xi), \label{eq:ricc_sigma}
\end{align}
\noindent with $\mathrm{cov}(\xi)$ the covariance matrix $\Exp \left[ (\xi - \Exp[\xi])(\xi - \Exp[\xi])^\top \right]$ of the random vector $\xi$.
\end{rem}

\paragraph*{Approximate solution via statistical linearization}
We assume that the mappings $f$, $g$, $G$ and $H$ are smooth enough (at least $C^1$) and satisfy the standard conditions guaranteeing the existence of strong solutions to the stochastic differential equations \eqref{eq:sde} and \eqref{eq:sde_obs}. Moreover, for every $t$ and $x$, we assume that the $(p \times p)$-matrix $\odiff(t,x)$ is invertible. At last, we assume that the cost satisfies the following hypothesis.
\begin{itemize}
  \item[(H)] The cost $J$ is quadratic (non necessarily homogeneous) in $x$ and $u$, namely,
\begin{equation}
\ell(t,x,u) = u^\top R(t) u + x^\top Q(t) x + q(t)^\top x + r(t)^\top u \quad \hbox{and} \quad \psi (x) = x^\top Q_\psi x  + q_\psi^\top x,
\end{equation}
where $R, Q, r, q$ are continuous mappings, and, for every $t \in [0,T]$, $R(t)$ is a positive definite $(m \times m)$-matrix, $Q(t)$ and $Q_\psi$ are semi-definite positive $(n\times n)$-matrices, $q(t), q_\psi$ belong to $\R^n$, and $r(t)$ to $\R^m$.
\end{itemize}

We will use an approximation given by the \emph{statistical linearization} (see \cite{Berret2020}) to derive a tractable yet approximately equivalent deterministic optimal control problem (see Theorem \ref{th:main} below). Let us recall what this means, adopting bold symbols to denote variables within the statistical linearization stage. Consider a process $\bf{x} = \{ \bf{x}_t \}$ solution of a stochastic dynamical systems of the form
\begin{equation}
d\bf{x}_{t}=\bf{f}(t, \bf{x}_{t},\bf{u}(t))\,dt+ \bf{G}(t,\bf{x}_{t},\bf{u}(t))\,d\bf{w}_{t},
\end{equation}
where $\bf{u}(\cdot)$ is an \textit{open-loop} control. The statistical linearization approach consists in approximating the mean and covariance of $\bf{x}_t$ by $\bf{m}(t)$ and $\bf{P}(t)$ respectively, which are the solutions of
\begin{equation}
\begin{cases}
\dot{\bf{m}}(t) & =\bf{f}(t,\bf{m}(t),\bf{u}(t)),\\
\dot{\bf{P}}(t) & = \frac{\partial\bf{f}}{\partial\bf{x}}(t,\bf{m}(t),\bf{u}(t))\bf{P}(t) + \bf{P}(t)\frac{\partial\bf{f}}{\partial\bf{x}}(t,\bf{m}(t),\bf{u}(t))^{\top} + \bf{G}(t,\bf{m}(t),\bf{u}(t)) \bf{G}(t,\bf{m}(t),\bf{u}(t))^{\top},
\end{cases}
\end{equation}
with $\bf{m}(0)=\Exp[\bf{x}_0]$ and $\bf{P}(0)=\mathrm{cov}(\bf{x}_0)$.


\begin{thm} \label{th:main}
For $u^\o  \in \mathcal{L}^2$, set
\begin{equation}
 \JJ(u^\o) = \psi(m^\o(T)) +  \mathrm{tr} ( Q_\psi \S(T)) + \int_0^T \left( \ell(t,m^\o,u^\o) +  \mathrm{tr} \left(  Q \S - \S C^\top (HH^\top)^{-1} C \S \E \right) \right) dt ,
\end{equation}
where  $(m^\o,\S,\E)(\cdot)$ is the trajectory of the following control system associated with $u^\o(\cdot)$:
\begin{equation}
\left\{
\begin{array}{l}
\dot m^\o  = \drift(t,m^\o,u^\o), \quad m^\o(0)=\Exp [\xi], \\
\dot \S  = A \S + \S A^\top - \S C^\top (H H^\top)^{-1} C \S + G G^\top, \quad \S(0) = \mathrm{cov}(\xi), \\
\dot \E  = - A^\top \E - \E A - \E BR^{-1}B^\top \E + Q, \quad \E(T) = -Q_\psi .
\end{array}
\right.
\end{equation}
Then, through statistical linearization of the process $ \bf{x} = (x,\hat{x}) $, the stochastic optimal control problem $\min_{u \in \mathcal{U}} \mathbb{J}(u)$ is approximated by the deterministic optimal control problem $\min_{u^\o  \in \mathcal{L}^2} \JJ (u^\o)$. Moreover, for any minimum ${u}^\o$ of $\JJ$, the policy $u=\pi({u}^\o,\KK)$ is a near-optimal control process for $\mathbb{J}$, where $\KK$ is the matrix-valued function associated with ${u}^\o$ defined by $\KK(t) = R(t)^{-1} B(t)^\top \E(t)$.
\end{thm}

The proof of Theorem \ref{th:main} is given in Methods. In summary, this theorem allows to replace the study of a partially observable stochastic optimal control problem by that of a deterministic optimal control problem in dimension $n(n+2)$, a far easier task. The accuracy of the approximation can be quantified precisely thanks to the error estimates of the statistical linearization given in \cite[Sect.~3.1]{Leparoux2024}.
Additionally, it is worth mentioning that the method---despite its name---fully preserves the nonlinear nature of the drift $\drift$ and, therefore, the controller's ability to exploit the nonlinear properties of the dynamics to plan optimal actions under uncertainty.

\begin{rem}\label{re:extensions}
The method can be extended beyond the assumptions used above.
\begin{enumerate}[label=(\roman*)]
  \item One of the main hypotheses of our framework is that the cost is quadratic (H). This assumption allows us to express $\mathcal{J}(u^\o,\K)$ solely as a function of the first two moments of $\bf{x}_t$, $\Exp[\bf{x}_t]$ and $\mathrm{cov}(\bf{x}_t)$, and to employ statistical linearization. For non-quadratic costs, one could approximate $\mathcal{J}(u^\o,\K)$ by replacing $\ell$ and $\psi$ with their quadratic expansions along $(m^\o,u^\o)$---similarly to DDP or iLQG---and then apply the remainder of the construction. An approximate solution could then be obtained from the resulting deterministic optimal control problem. However, this heuristic approach does not provide a theoretical justification for the approximation, nor does it yield an error estimate such as that given in \cite{Leparoux2024}. In contrast, our method applies without modification to cost functions composed of quadratic terms in $(x_t, u_t)$ together with non-quadratic terms depending only on $\Exp[x_t]$ and $\Exp[u_t]$. Consequently, more complex cost functions can be implemented to shape the system's average behavior, for example a minimum hand jerk model for a multi-joint arm \citep{Flash1985}.

  \item Another limiting assumption of our framework is that the use of stochastic controls of the form $u=\pi(u^\o,\K)$ requires unbounded controls with values in $\R^m$ whereas in most applications controls are bounded. As explained in~\cite[Sect.~4.3.2]{Leparoux2024}, a first approach to circumvent this limitation consists of inserting saturations on the control directly in the dynamics. In this way, $u$ can be considered as an unbounded input, while the actual control of the system is the saturation of $u$. A second approach is to add constraints using $u^\o$ and $\P$ (defined by~\eqref{eq:P}; see Methods) to restrict deviations from the mean by a given amount of standard deviation. Note that this solution would require adding $\P$ as a variable into the augmented deterministic problem. A third approach is to only set bounds on $u^\o$ and test the controller \textit{a posteriori} by clamping $u_t$ to reflect hard physical/biological constraints. A similar treatment can be applied to the state variable $x$, including if path constraints and not only box constraints, are considered (see \cite{VanWouwe2022} for an example).

\item The statement of Theorem~\ref{th:main} also holds when the original stochastic control problem has a free final time. In that case, the corresponding approximated deterministic optimal control problem is also in free time. However, in that setting, the hypothesis on the cost function must be modified to ensure the existence of solutions to both optimal control problems. In hypothesis (H), we can typically add a positive term $c(t)$ to the running cost $\ell$, with $c(t) \ge \alpha > 0$ for all $t \ge 0$. This framework coincides with the free-time formulations in~\citep{Berret2016}, for example.
\end{enumerate}
\end{rem}

\paragraph*{Modeling sensory feedback delay}
\label{se:delays}
An important motivation for the above-developed framework is modeling the neural control of movement, which is characterized by significant sensory feedback delays. We explain below how to integrate those delays in the framework. Let us assume that the physical state of a system at time $t$ (e.g., position, velocity, muscle activation etc. for the human sensorimotor system) is described by an It\^{o} process $s=\{\s_t\}$, which is the solution to a stochastic control system of a form similar to \eqref{eq:sde}:
\begin{equation} 
    d\s_t = \tilde{\drift}(t,\s_t,u_t) dt + \tilde{\diff}(t,\s_t,u_t) d\tilde{w}_t, \quad \s_0= \tilde{\xi}.
\end{equation}
The sensory signal at time $t$ is a vector $\tilde{\odrift}(\s_t) \in \R^p$ that depends on the physical state of the system. This signal is subject to transmission latency, integration time and neural noise prior to reaching the central controller, resulting in a feedback signal with significant delay and reduced signal-to-noise ratio. We denote this ``integrated'' sensory feedback by $z_t \in \R^p$. The delay introduced by this process depends on the specific sensory modality involved. For instance, proprioceptive feedback is significantly faster than visual feedback in the human sensorimotor system (e.g., 50\,ms versus more than 100\,ms). To model delayed sensory feedback, we thus assume that the process $z=\{z_t\}$ satisfies the following stochastic differential equation
\begin{equation} 
    dz_t = D^{-1} \left(\tilde{\odrift}(\s_t) -z_t \right) dt + \diff^z(t,\s_t) dw^z_t, \quad z_0=\tilde{\odrift}(\s_0),
\end{equation}
where $D$ is the diagonal matrix $\mathrm{diag}(\Delta_1, \dots, \Delta_p)$, and each $\Delta_i>0$ is the typical delay introduced by each sensory-specific integration process. The diffusion term, $\diff^z(t,\s_t) dw^z_t$, accounts for all the internal noise affecting this process.

By further considering that sensory feedback $z_t$ is affected by observation noise, the observed state is modeled as a stochastic process $y=\{y_t\}$ that is a solution to
\begin{equation} 
    dy_t = z_t dt + \odiff(t,\s_t,z_t) dv_t, \quad y_0=0.
\end{equation}
where $\odiff$ specifies the noise covariance for the different sensory modalities.

Defining the extended state $x_t=(\s_t,z_t)$, we therefore obtain a model under the form of a stochastic control system as in \eqref{eq:sde}, where
\begin{align}
 \drift(t,x_t,u_t) = \left(\tilde{\drift}(t,\s_t,u_t),D^{-1} \left(\tilde{\odrift}(\s_t) -z_t \right) \right), &\quad
 \diff(t,x_t,u_t) = \left(
    \begin{array}{cc}
        \tilde{\diff}(t,\s_t,u_t) & 0 \\
            0 & \diff^z(t,\s_t) \\
    \end{array}
\right), \\
 \quad dw_t= \left(d\tilde{w}_t, dw^z_t \right), &\quad \xi = \left( \tilde{\xi},\tilde{\odrift}(\tilde{\xi}) \right),
\end{align}
with an observation equation of the form of \eqref{eq:sde_obs} with $\odrift(t,x_t) = z_t$.

Note that the state estimate $\hat{x}$ needed by the optimal feedback control policy $u$ is eventually obtained from multisensory integration \citep{Ernst2004} and internal forward model predictions \citep{Wolpert1995c}, as implemented by the linearized Kalman filter (see \eqref{eq:LKF}).

\subsection*{Application to human sensorimotor control}


\paragraph*{Stabilization task with two muscles}

We first consider the human forearm stabilization task \citep{Hogan1984a}. The task consists of stabilizing the human forearm in its unstable equilibrium position for 5 seconds by controlling a pair of antagonist muscles with variable viscoelasticity properties. The activations of the flexor and extensor muscles are denoted by $u_1$ and $u_2$ respectively. These muscle activations can vary joint net torque via reciprocal activations ($\propto(u_1-u_2)$) and joint viscoelasticity via simultaneous activations ($\propto(u_1+u_2)$). Proprioceptive feedback and visual feedback about the joint position are assumed to be available to stabilize the system, and we assume they are subject to uncertainties and latencies. Hereafter, the subscripts $p$ and $v$ refer to proprioception and vision respectively. Four parameters are introduced to vary the magnitudes of sensory delay (denoted by $\rho_p$ and $\rho_v$) and sensory noise (denoted by $\eta_p$ and $\eta_v$). The problem is to find the optimal strategy, balancing feedforward and feedback muscle activations, to stabilize this nonlinear, partially observable stochastic system. Here optimality is meant in terms of minimizing the expectation of an integral quadratic cost composed of squared muscle activations and squared task errors. Details are about the simulations and models are given in Methods.

We investigated how the optimal control policy depends on the $\rho$ and $\eta$ parameters, by varying them through the following scaling factors: 0.2x, 0.5x, 2x and 5x where the reference level (1x) corresponds to the following values: $\rho_p=0.05\,s$ and $\rho_v=0.10\,s$;  $\eta_p=5\,^\circ\sqrt{s}$ and $\eta_v=1\,^\circ\sqrt{s}$. In other words, proprioceptive feedback is thus assumed to be faster but less precise than visual feedback.

The evolution of the optimal strategy with respect to $\eta$ and $\rho$ parameters is depicted in Fig.\,\ref{fig:1dof_evolution}. The main observation is that the optimal strategy significantly depends on the level of sensory noise. One can see that the optimal stiffness ($\propto(u_1+u_2)$) increases with $\rho$ values, whereas the feedback gain decreases accordingly (Fig.\,\ref{fig:1dof_evolution}A,B where the mean value of $\K_{1,1}$ is depicted). This shows that, with large sensory noise, the optimal strategy is to correct errors mostly by exploiting the variable mechanical impedance of muscles via their co-contraction. In contrast, with small sensory noise, it is better to rely more on locally optimal state feedback control. To further investigate the optimal strategy, we quantified the amount of absolute net torque used to correct errors during task execution. Interestingly, with large sensory noise, less net torque is produced by muscles (i.e., $\propto(u_1-u_2)$) because muscle co-contraction contributes strongly to the intrinsic system's stability (Fig.\,\ref{fig:1dof_evolution}C). With low sensory noise, it is in contrast optimal to reduce task errors via net torque corrections stemming from reciprocal muscle activations and state feedback control. Finally, the optimal expected cost increases with the amount of sensory noise $\eta$ (Fig.\,\ref{fig:1dof_evolution}D), in particular because muscle co-contraction is a costly---yet optimal---strategy when facing large sensory noise. This illustrates that the proposed framework predicts an optimal balance between feedforward and feedback control to stabilize the system, which fundamentally depends on the quality of state feedback. The same trends were observed when scaling sensory delays via $\rho_p$ and $\rho_v$.

In summary, with large sensory noise or delay, our simulations demonstrate that it is optimal to increase feedforward muscle co-contraction and reduce state feedback dependency.

\begin{center}
\begin{figure}[H]
\centering{}
\centerline{\includegraphics[width=11cm]{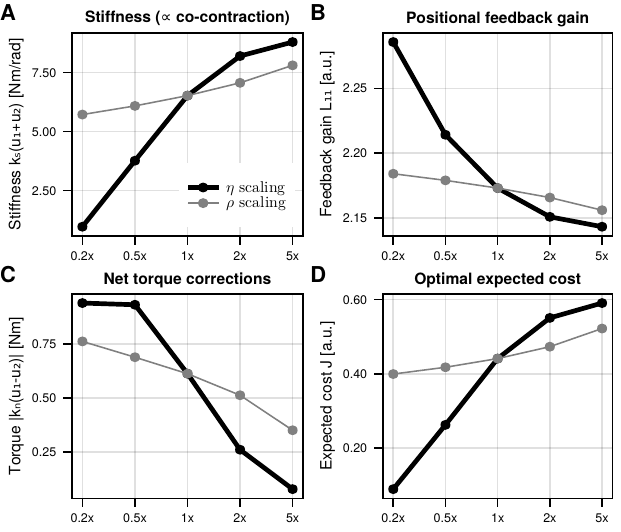}}
\caption{\label{fig:1dof_evolution} Optimal strategy depending on sensory noise and delays. A: Evolution of mean stiffness with respect to variations of sensory noise and delay (average value of $k_s(u_1+u_2)$ across 100 samples). B: Mean positional control feedback gain with respect to variations of sensory noise and delays ($\K_{1,1}(t)$ is depicted). C: Mean absolute net torque computed (average value of $k_n(u_1-u_2)$ across 100 samples). D: Optimal expected cost with respect to variations of sensory noise and delay. Default values for parameters, corresponding to the 1x condition, were as follows: $\rho_p=0.05\,s$, $\rho_v=0.10\,s$, $\eta_p=5.0\,^\circ\sqrt{s}$ and $\eta_p=1.0\,^\circ\sqrt{s}$.}
\end{figure}
\par\end{center}

In Fig.\,\ref{fig:1dof_trajectories}, the corresponding optimal strategies are illustrated for different sensory noise levels. Task performance is depicted in Fig.\,\ref{fig:1dof_trajectories}A,B,C (joint position and velocity) and the optimal motor commands are depicted in Fig.\,\ref{fig:1dof_trajectories}D,E,F (net torque and muscle activations). With low sensory noise (0.2x), stabilization is mostly ensured by state feedback corrections, which generates substantial net torques to correct task errors through reciprocal muscle activations. It can be noted that a small co-contraction is optimal in this condition. With the default sensory noise (1x), a moderate co-contraction becomes the optimal strategy but feedback corrections are still preponderant as revealed by the large net torque fluctuations. This means that the feedforward stiffness is insufficient to effectively stabilize the system but an optimal value exist given the reduced efficiency of state feedback control in this case. For large sensory noise, feedback control becomes mostly irrelevant so that a substantial co-contraction stabilizing the system almost entirely becomes necessary. Indeed, the optimal control policy relies much less on feedback control as revealed by the drastic reduction of state feedback corrections via net torque production.

\begin{center}
\begin{figure}[H]
\centering{}
\centerline{\includegraphics[width=13cm]{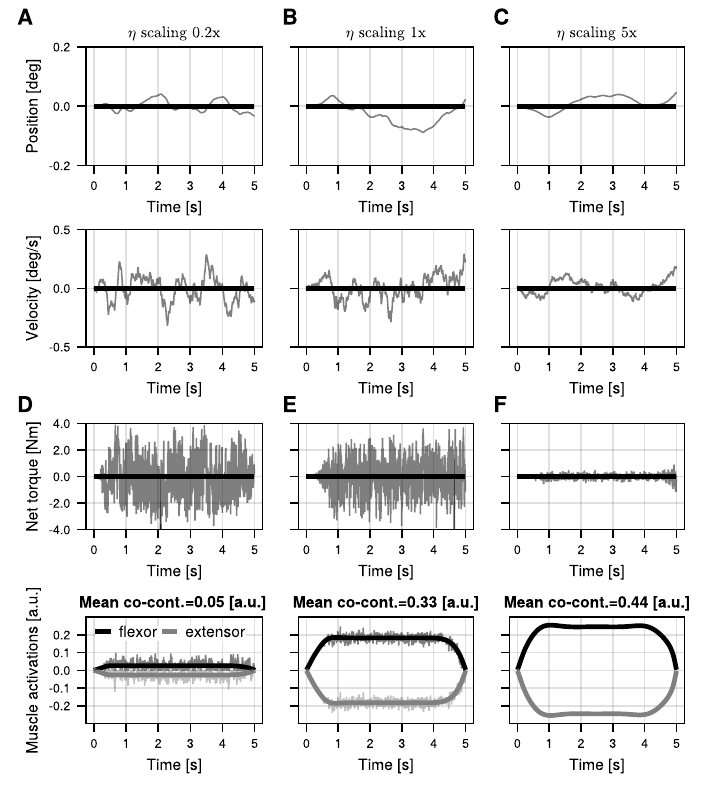}}
\caption{\label{fig:1dof_trajectories} Optimal trajectories for different sensory noise magnitudes. A: Position and velocity traces for a sensory noise of 0.2x the reference level. Thick traces depict the mean behavior and thin traces depict a single trial. B,C: Same information for the different noise magnitudes, 1x and 5x. D: Net torque and muscle inputs for a sensory noise of 0.2x the reference level. Flexor torques are depicted with positive values and extensor torques are depicted with negative values for clarity. E,F: Same information for the different noise magnitudes, 1x and 5x. The mean co-contraction, defined as the mean of $u_{1}+u_{2}$ across the trial, is reported above the graphs.}
\end{figure}
\par\end{center}

\paragraph*{Planar reaching task with six muscles}

Here we simulate planar reaching movements in the forward direction as it has been done in several experimental works \citep{Burdet2001,Franklin2003,Franklin2004,Franklin2007,Franklin2023}. In this task, the endpoint path is approximately straight with a bell-shaped velocity profile, which resembles a minimum jerk trajectory \citep{Flash1985}, while a force field can be applied at the hand level. The human arm is modeled as a two-link arm actuated by 6 muscles with activation-dependent viscoelasticity. This model provides a phenomenological description of variable impedance muscles and is described in \cite{Katayama1993}. For the simulations, the physical state of the system consists of 10 dimensions, comprising the shoulder and elbow joint positions and velocities and the 6 muscle activations (which are subject to a first-order activation dynamics with response time of $0.06\,s$). We assume that sensory signals give information about the shoulder and elbow positions with a delay of $\rho=0.10\,s$, therefore adding 2 dimensions ($z_t$ variable). The augmented state is of dimension $n=12$, which means that the equivalent deterministic optimal control problem is of dimension $n(n+2)=168$. Sensory noise $\eta$ was assumed to be equal to $0.05\,^\circ\sqrt{s}$ when both vision and proprioception are combined (same value for both joints). To simulate movements in the absence of vision, we set $\eta = 5.0\,^\circ\sqrt{s}$, modeling the lack of visual feedback as a significant increase in observation noise. To simulate motor adaptations in a divergent force field (DF), i.e. where muscle co-contraction has been proved to play a critical role, we model an external force field of the form $F_e=(\beta \mathrm{x}, 0)$ (acting along the $\mathrm{x}$-axis in the Cartesian plane) with $\beta = 300 \, N/m$. The null field (NF) condition corresponds to $\beta=0 \, N/m$. Details about the simulations and models are given in Methods.

\begin{center}
\begin{figure}[H]
\centering{}
\centerline{\includegraphics[width=14cm]{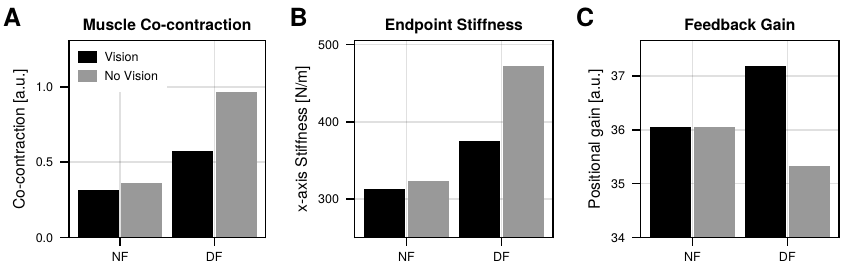}}
\caption{\label{fig:2dof_evolution} Optimal strategy depending on sensory noise and force field magnitudes. A. Optimal muscle co-contraction. Co-contraction was calculated by averaging values across muscle groups, which were organized into three antagonist pairs (shoulder, elbow, and biarticular) and across time. B. Optimal lateral endpoint stiffness. This stiffness along the $\mathrm{x}$-axis was averaged along the trajectory, using the joint stiffness and arm's Jacobian matrix to compute it. C. Optimal feedback gain. Only the mean positional gain averaged across the 6 muscles and 2 joints and across time is depicted. For all variables, grand mean values for the NF and DF conditions are reported, both with and without vision.}
\end{figure}
\par\end{center}

In Fig.\,\ref{fig:2dof_evolution}, we show how the optimal control strategy changes between the NF and the DF conditions, both with and without vision. First, we observe that muscle co-contraction (Fig.\,\ref{fig:2dof_evolution}A) is smaller in NF than in DF, in agreement with empirical findings \citep{Franklin2003}. In NF, co-contraction is marginally larger without vision than with vision \citep{Franklin2007}. In DF, muscle co-contraction is substantially stronger without vision than with vision. Hence, the absence of vision in DF leads to the largest co-contraction level. Similar trends are observed for the mean endpoint stiffness along the lateral direction (Fig.\,\ref{fig:2dof_evolution}B), in agreement with empirical findings \citep{Franklin2007}. Regarding the mean positional feedback gains (Fig.\,\ref{fig:2dof_evolution}C), their values are very similar for both visual conditions in NF. In contrast, the optimal feedback gain in DF is increased with vision but reduced without vision compared to NF level. These results show that the interaction between task dynamics (DF versus NF) and sensory noise (with vision versus without vision) has non-trivial effects on the feedforward-feedback nature of the optimal control policy.

The corresponding optimal trajectories are depicted in Fig.\,\ref{fig:2dof_trajectories}. The endpoint trajectories show that the task is executed relatively well in all conditions (Fig.\,\ref{fig:2dof_trajectories}A). This similar performance is however achieved using a variety of distinct control strategies (Fig.\,\ref{fig:2dof_trajectories}B). With vision (low sensory noise), only a moderate level of co-contraction is observed in the NF condition. Without vision, co-contraction only increases slightly, mainly at the shoulder level. In the DF condition, the task is more challenging because it is unstable. Non-adapted strategies (i.e., using the control policy of NF in DF) lead to diverging trajectories that miss the target (green traces in Fig.\,\ref{fig:2dof_trajectories}A). As a consequence, even when vision is available, more muscle co-contraction is necessary in DF to ensure robust task performance. Notably, the co-contraction of biarticular muscles starts to play an important role in the DF condition, a phenomenon that has been observed experimentally \citep{Franklin2003} (co-contraction at the shoulder also increases but little changes are observed at the elbow level). Interestingly, substantial feedback corrections are still visible in muscle inputs when vision is available. Without vision (or equivalently with large sensory noise) in DF, co-contraction attains its largest level, with some muscle inputs saturating near the end of the movement where the arm is more extended. Remarkably, with such a large muscle co-contraction, very little feedback muscle inputs are visible in single-trial traces. This means that the feedforward control is robust enough so that online motor corrections are not triggered (note that feedback gains are also reduced in this case).

\begin{center}
\begin{figure}[H]
\centering{}
\centerline{\includegraphics[width=13.5cm]{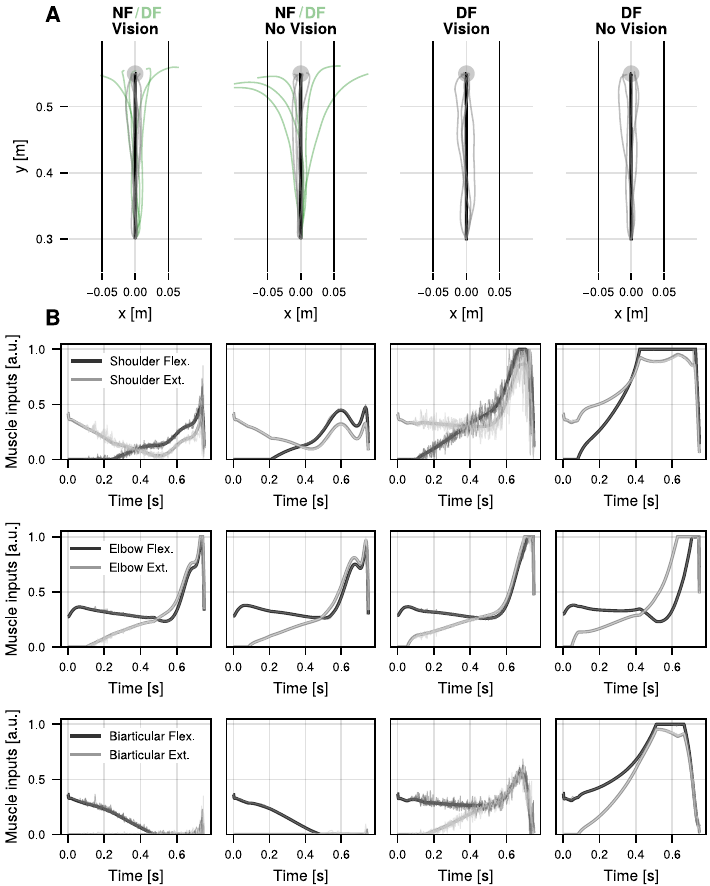}}
\caption{\label{fig:2dof_trajectories} Optimal trajectories for different sensory noise and force field magnitudes. A. Endpoint trajectories (mean and 5 samples depicted) in NF and DF conditions, both with and without vision. Green traces depict trajectories in DF but using the optimal control policy of the NF condition (i.e., a non-adapted scenario). B. Optimal muscle inputs (mean and same 5 samples depicted) for shoulder, elbow and biarticular muscles in the NF and DF conditions, with or without vision.}
\end{figure}
\par\end{center}

\section*{Discussion}

This work introduces a continuous-time framework for solving partially observable nonlinear stochastic optimal control problems in relatively high dimensions. By integrating planning under uncertainty with state estimation and feedback control, the framework enables control policies to exploit nonlinear task dynamics while accounting for sensorimotor uncertainties and latencies. Through statistical linearization, the original stochastic problem is transformed into a tractable deterministic equivalent that preserves the essential structure of the original formulation while significantly reducing computational complexity. Notably, this framework maintains a level of interpretability and delivers theoretical guarantees often absent in contemporary data-driven approaches, and remains solvable with state-of-the-art numerical tools for deterministic optimal control.

When applied to human neuromechanics, our approach provides a principled explanation for how the CNS balances feedforward and state feedback control to ensure stable behaviors. Specifically, the framework naturally predicts the emergence of muscle co-contraction as a necessary complement to high-level feedback control. Our results show that increased sensory noise and latencies trigger a shift of the control strategy toward greater reliance on muscle co-contraction and reduced feedback dependency---a finding that aligns with empirical observations of motor adaptation in uncertain environments \citep{Franklin2003, Franklin2007}. This trade-off is particularly critical in unstable tasks, where the consequences of sensory latency and uncertainty on task performance significantly alter the nature of the optimal control policy. By reconciling the long-standing debate of feedforward impedance control versus state feedback control \citep{Scott2012, Burdet2013}, this framework unifies a vast body of literature---from classical minimum hand jerk or variance models \citep{Flash1985,Harris1998} to modern optimal feedback control \citep{Todorov2002}.

While the current formulation handles smooth nonlinearities, future work should explore how to model tasks involving contacts or discontinuities such as walking. Additionally, while we demonstrate scalability to relatively high-dimensional systems, the augmented state space grows significantly with the inclusion of covariance matrices, potentially challenging the limits of standard solvers for systems with numerous degrees of freedom. Efficient tools for solving this class of problems, exploiting the structure of differential equations and the underlying sparsity of the problem, should be developed. Nonetheless, the framework’s compatibility with inverse optimal control and free-time formulations \citep{Berret2011b, Berret2016} already opens new avenues for identifying the cost functions underlying human movement, moving beyond their traditional focus on mean behavioral trajectories. Critically, our framework is not limited to viscoelasticity modulation; for rigid body dynamics, it can exploit the configuration-dependent inertia matrix during planning, offering a holistic approach to mechanical impedance control that accounts for the presence of feedback. The versatility of the framework further allows considering environmental uncertainty (e.g., parametric randomness) similar to \citep{Berret2024}. By making new testable predictions, this approach will thus advance our understanding of the sensorimotor control strategies employed by the CNS, which universally contend with noise and delays.

Beyond biological movement modeling, the generality of this continuous-time formulation makes it a valuable tool for engineered systems. It offers a theoretical and practical solution for designing robust control strategies in complex nonlinear robotic systems, especially those equipped with variable impedance actuators \citep{Vanderborght2012}. Ultimately, the present stochastic optimal feedforward-feedback control framework provides a unified computational language to study how robust behavior is synthesized in any nonlinear system where both sensing and acting are fundamentally affected by noise or delays.

\section*{Methods}

\paragraph*{Proof of the main theoretical result}
\begin{proof}[Proof of Theorem~\ref{th:main}]%
    Let $u^\o  \in \mathcal{L}^2$, $\K \in \mathcal{L}_M^2$, and $u=\pi (u^\o, \K)$. Since $J(u)$ is quadratic, we can write $\mathcal{J}(u^\o,\K)= \mathbb{J} \left( \pi (u^\o,\K) \right)$ in terms of the mean and covariance of $x_t$ and $\hat{x}_t$. Specifically, denote by $m_x(t)=\Exp[x_t]$, $m_{\hat{x}}(t)=\Exp[\hat{x}_t]$ the mean, and by $\mathrm{cov}(x_t)$, $\mathrm{cov}(\hat{x}_t)$ the covariance matrices of $x_t$ and $\hat{x}_t$ respectively. Then
\begin{multline}
\mathcal{J}(u^\o, \K) = m_x(T)^\top Q_\psi m_x(T)  + q_\psi^\top m_x(T) + \mathrm{tr} \left( Q_\psi \mathrm{cov}(x_T) \right) +
  \int_0^T (u^\o+\K(m_{\hat{x}}-m^\o))^\top R (u^\o+\K(m_{\hat{x}}-m^\o)) \, dt \\
 + \int_0^T \left( \mathrm{tr} \left( \K^\top R \K \mathrm{cov}(\hat{x}_t) \right) + m_x^\top Q m_x + \mathrm{tr} \left( Q \mathrm{cov}(x_t)\right) + q^\top m_x + r^\top u^\o + r^\top \K (m_{\hat{x}}-m^\o) \right) \, dt.
\end{multline}

Let us apply the statistical linearization approach to $\bf{x}=(x,\hat{x})$ associated with the control $u=\pi (u^\o,\K)$, which is solution of the stochastic dynamical system (for the sake of readability we remove when possible the dependence w.r.t.\ time),
\begin{multline}
\left(
  \begin{array}{c}
    dx_t \\
    d \hat{x}_t \\
  \end{array}
\right) =
\left(
  \begin{array}{c}
    \drift(t,x_t,u^\o + \K ( \hat{x}_t - m^\o)) \\
    \drift(t,\hat{x}_t,u^\o + \K ( \hat{x}_t - m^\o)) + \L (\odrift(t,x_t) - \odrift(t,\hat{x}_t)) \\
  \end{array}
\right) dt  \\
+ \left(
  \begin{array}{cc}
    \diff(t,x_t,u^\o + \K ( \hat{x}_t - m^\o)) & 0 \\
    0 & \L \odiff(t,x_t) \\
  \end{array}
\right)
\left(
  \begin{array}{c}
    dw_t \\
    dv_t \\
  \end{array}
\right),
\end{multline}
with $\bf{u}=(u^\o,\K)$ and $\bf{w}_t=(w_t,v_t)$.

Thus we approximate the mean $\Exp[\bf{x}_{t}]=(m_x,m_{\hat{x}})(t)$ by $\bf{m}(t)=(m, \hat{m})(t)$ solution of
\begin{align}
  \dot m & = \drift(t,m,u^\o + \K ( \hat{m} - m^\o)), \qquad m(0)= m_x (0) = \Exp[\xi],\\
  \dot{\hat{m}} & = \drift(t,\hat{m},u^\o + \K ( \hat{m} - m^\o)) + \L (\odrift(t,m) - \odrift(t,\hat{m})) , \qquad \hat{m}(0)= m_{\hat{x}} (0) = \Exp[\xi].
\end{align}
From \eqref{eq:moyenne} we immediately deduce $m(t)= \hat{m}(t)=m^\o (t)$ for every $t\in [0,T]$. Therefore,
\begin{equation}
 \mathrm{cov}(\bf{x}_t) =
\left(
 \begin{array}{cc}
     \mathrm{cov}(x_t) &  \mathrm{cov}(x_t,\hat{x}_t) \\
     \mathrm{cov}(\hat{x}_t,x_t) &  \mathrm{cov}(\hat{x}_t) \\
  \end{array}
\right) \quad \hbox{is approximated by } \quad \bf{P}(t)=
\left(
  \begin{array}{cc}
    P_{11}(t) & P_{12}(t) \\
    P_{12}(t)^\top & P_{22}(t) \\
  \end{array}
\right)
\end{equation}
solution of the differential equation
\begin{equation}
\dot{\bf{P}} = \left( \begin{array}{cc} A & B\K \\ \L C & A+B\K-\L C \\ \end{array} \right)\bf{P}
+ \bf{P} \left( \begin{array}{cc} A^\top & (\L C)^\top \\ (B\K)^\top & (A+B\K-\L C)^\top \\ \end{array} \right)
+ \left( \begin{array}{cc} G & 0 \\ 0 & \L H \\ \end{array} \right)  \left( \begin{array}{cc} G^\top & 0 \\ 0 & (\L H)^\top  \\ \end{array} \right),
\end{equation}
with $P_{11}(0)=\mathrm{cov}(\xi)$, $P_{12}(0)=P_{22}(0)=0$. Using the definition of $\L=\S C^\top (HH^\top)^{-1}$, we can check directly by replacing in the equation above that $\bf{P}$ satisfies
\begin{equation}\label{eq:Pformula}
P_{11}(t)=\P(t) + \S(t) \quad \hbox{and} \quad P_{12}(t)=P_{22}(t) = \P(t), \quad \forall t \in [0,T],
\end{equation}
where $\S(\cdot)$ is the solution of the Riccati equation of the linearized Kalman filter defined by~\eqref{eq:ricc_sigma}, and $\P(\cdot)$ is the solution of the following differential equation,
\begin{equation}\label{eq:P}
\dot \P = (A+B\K)\P + \P (A+B\K)^\top + \S C^\top (HH^\top)^{-1} C \S, \qquad \P(0)=0.
\end{equation}

\begin{rem}\label{re:lien_KF}
Formula~\eqref{eq:Pformula} shows that the statistical linearization of $\bf{x}=(x,\hat{x})$ captures the properties of the Kalman filter. Indeed, for the standard Kalman-Bucy filter in the linear case, there holds $\mathrm{cov}(\hat{x}_t - x_t) = \S(t)$ (error estimation covariance) and $\Exp[(\hat{x}_t - x_t)\hat{x}_t^\top] = 0$ (projection theorem). In our framework, $\Exp[(\hat{x}_t - x_t)\hat{x}_t^\top]$ is approximated by $P_{22}-P_{12}$, so $P_{22}=P_{12}$ reflects the projection theorem. And $\mathrm{cov}(\hat{x}_t - x_t)$ is approximated by $P_{22}+ P_{11} - P_{12}-P_{12}^\top= P_{11} - P_{22}$, so $P_{11} = P_{22} + \S$ reflects that $\S$ is the error estimation covariance.
\end{rem}\smallskip

Replacing the mean and covariance of $x_t$ and $\hat{x}_t$ by their statistical linearization approximations in the expression of $\mathcal{J}(u^\o,\K)$, we obtain the following approximation of the latter,
\begin{align}
\mathcal{J}_{app}(u^\o,\K) & = m^\o(T)^\top Q_\psi m^\o(T)  + q_\psi^\top m^\o(T) + \mathrm{tr} \left( Q_\psi (\P(T)+\S(T)) \right)
   \\
 & \qquad \quad + \int_0^T \left( (u^\o)^\top R u^\o + \mathrm{tr} \left(\K^\top R \K \P \right) + (m^\o)^\top Q m^\o + \mathrm{tr} \left( Q (\P+\S))\right) + q^\top m^\o + r^\top u^\o  \right) dt \\
 & = J(u^\o) + \Delta J(u^\o,\K),
\end{align}
where $J(u^\o)= \psi(m^\o(T)) + \int_0^{T} \ell(t,m^\o,u^\o) \, dt$ and
\begin{equation}
\Delta J(u^\o,\K) = \mathrm{tr} \left( Q_\psi (\P(T)+\S(T)) \right) + \int_0^T \left( \mathrm{tr} \left(\K^\top R \K \P \right) + \mathrm{tr} \left( Q (\P+\S) \right) \right) dt.
\end{equation}

Let us reinterpret the term $\Delta J(u^\o,\K)$. Given an open-loop control $u^\o (\cdot) \in \mathcal{L}^2$, let us introduce the approximated quadratic cost
\begin{equation}
\mathbb{J}^{\lq}(\uu) = \Exp [J^{\lq}(\uu)], \qquad \hbox{with} \quad  J^{\lq}(\uu)= \x_T^\top Q_\psi \x_T  + \int_0^{T} \left( \uu^\top R(t) \uu + \x^\top Q(t) \x \right) dt,
\end{equation}
where $\x_t$ is the solution of the linearized system around $(m^\o(\cdot),u^\o(\cdot))$ associated with $\uu$, namely,
\begin{equation}
\left\{
  \begin{array}{l}
        d \x_t = \left( A(t) \x_t + B(t) \uu_t \right) dt + G(t) dw_t, \quad \x_0 \sim \mathcal{N} (0, \mathrm{cov}(\xi)), \\
    d \y_t = C(t) \x_t dt + H(t) dv_t, \quad \y_0=0.
  \end{array}
\right.
\end{equation}
From classical LQG theory (see \cite{Davis1977} for example), the minimum of $\mathbb{J}^{\lq}$ is attained at the control $\uu = \KK  \hx$, where:
\begin{itemize}
  \item the matrix-valued function $\KK(\cdot)$ (which implicitly depends on $u^\o$) is defined by
\begin{align}
  & \KK(t) = R(t)^{-1} B(t)^\top \E(t) \qquad \hbox{with $\E(\cdot)$ the solution of} \label{eq:LQgain} \\
  & \dot \E = - A^\top \E - \E A - \E BR^{-1}B^\top \E + Q, \quad \E(T) = -Q_\psi ; \label{eq:LQRicc}
\end{align}
  \item $\hx_t$ is the solution of the Kalman-Bucy filter,
\begin{equation}\label{eq:KB}
d \hx_t = \left( A(t) \hx_t + B(t) \uu_t \right) dt + \L(t) \left( d\y_t - C(t) \x_t dt \right), \quad \hx_0=0 ,
\end{equation}
the matrices $\L(t)$ being defined by equations~\eqref{eq:LKFgain}.
\end{itemize}

Now, it follows from the properties of the Kalman-Bucy filter (see Remark~\ref{re:lien_KF}) that $\Delta J(u^\o,\K)= \mathbb{J}^{\lq} (\K \hx)$. Since the minimal value of $\mathbb{J}^{\lq}$ is $\mathbb{J}^{\lq}(\KK  \hx)$, we have
\begin{equation}
  \min_{\K  \in \mathcal{L}_M^2} \Delta J(u^\o,\K) = \Delta J(u^\o,\KK) \quad \hbox{and} \quad \JJ(u^\o) := \min_{\K  \in  \mathcal{L}_M^2} \mathcal{J}_{app}(u^\o, \K) = \mathcal{J}_{app}(u^\o, \KK).
\end{equation}
As a consequence, through statistical linearization of  $ \bf{x} = (x,\hat{x}) $,  the minimal value $\min_{u \in \mathcal{U}} \mathbb{J}(u)$ is approximated by  $\min_{u^\o  \in \mathcal{L}^2} \JJ (u^\o)$.

Finally, we can remove $\P$ from the expression of $\JJ(u^\o)$ as follows. We first expand the derivative $\frac{d}{dt}(\E \P)$ by using the expressions of $\dot \P$ and $\dot \E$. Then, taking into account the definition of $\KK$ and the properties of the trace operator, we obtain
\begin{equation}
\frac{d}{dt} \mathrm{tr} (\E \P) =  \mathrm{tr} \left({\KK}^\top R \KK \P + Q  \P + \S C^\top (HH^\top)^{-1} C \S \E \right) .
\end{equation}
On the other hand, from the initial conditions of $\E$ and $\P$ we deduce $\mathrm{tr}( Q_\psi \P(T)) =- \int_0^T \frac{d}{dt} \mathrm{tr} (\E \P) \,dt$. Inserting this term in the expression of $\Delta J(u^\o,\KK)$, we obtain
\begin{equation}\label{eq:deterministic+LQG}
 \JJ(u^\o) = J(u^\o) + \mathrm{tr} ( Q_\psi \S(T)) + \int_0^T \left( \mathrm{tr} \left(  Q \S - \S C^\top (HH^\top)^{-1} C \S \E )\right) \right) dt.
\end{equation}

\begin{rem}
In this expression, $\JJ(u^\o)$ appears as the sum of the cost $J(u^\o)$ of the deterministic optimal problem and the optimal cost of the LQG problem defined by linearization around $(m^\o,u^\o)$ (see \cite[Th.~5.3.3]{Davis1977}\footnote{There is a small mistake in the formula in \cite[Th.~5.3.3]{Davis1977},  an erroneous multiplication by an extra term $(GG^\top)^{-1}$. The error comes from the expression of $\tilde{C}$ in page 179, which should be multiplied by $G$ since $\nu$ is not normalised.}). This is reminiscent of the \emph{assumed certainty equivalence design} or \emph{neighboring-optimal stochastic control} \citep{Bryson1969,Athans1971,maybeck1982stochastic,stengel1986optimal}. However, in the latter methods, the control $u^\o$ is chosen beforehand as a solution to the deterministic optimal control problem defined by the dynamics~\eqref{eq:moyenne} of the mean and the cost $J(u^\o)$. In contrast, in our method, the choice of the control $u^\o$ is made \textit{a posteriori} by minimizing the sum of $J(u^\o)$ and the LQG cost.
\end{rem} \smallskip

Thus, $\JJ(u^\o)$ writes as a function of the values of $u^\o(\cdot)$, $m^\o(\cdot)$, $\S (\cdot)$, and $\E(\cdot)$. The last two are solutions of the Riccati differential equations~\eqref{eq:LKFgain} and~\eqref{eq:LQRicc} respectively. Since the matrices $A(t)$, $B(t)$, $C(t)$, $G(t)$ and $H(t)$ appearing in these equations are function of $(m^\o(t),u^\o(t))$, equations~\eqref{eq:moyenne}, \eqref{eq:LKFgain}, and~\eqref{eq:LQRicc} write as a control system
\begin{equation}\label{eq:syst_global}
\frac{d}{dt}(m^\o,\S,\E)(t)= F(t,(m^\o,\S,\E)(t), u^\o(t)),
\end{equation}
whose control $u^\o$ belongs to $\R^m$ and whose state $(m^\o,\S,\E)$ belongs to $\R^{n+n(n+1)}$ (recall that $\S$ and $\E$ are $(n \times n)$ symmetric matrices). Hence the minimization of $\JJ(u^\o)$ is a standard deterministic optimal control problem in finite dimension.
\end{proof}

\paragraph*{Simulation details for the 1-dof stabilization task}

We consider a phenomenological model of variable impedance muscles, similarly to \cite{Hogan1984a}.
The flexor and extensor torques generated at the elbow by the biceps and triceps are given by:
\begin{equation}
\left. \begin{array}{rlll}
\textrm{Flexor torque:} & \tau_{1,t} & = & u_{1,t}(k_{n}-k_s \theta_t - k_{d} \dot{\theta}_t)\\
\textrm{Extensor torque:} & \tau_{2,t} & = & -u_{2,t}(k_{n}+k_s \theta_t + k_{d} \dot{\theta}_t)
\end{array}\right.\label{eq:torques}
\end{equation}


In this task, the state vector is $s_t=(\theta_t, \dot{\theta}_t)$ where the elements are respectively the joint angle and velocity, and the control vector corresponds to muscle activations, $u_t=(u_{1,t},u_{2,t})$, where $u_{1,t}\in [0,1]$ and $u_{2,t}\in [0,1]$ (in arbitrary units, a.u.) allow to regulate both net torque and viscoelasticity at the joint.

The system's drift writes as follows:
\begin{equation}
\tilde{f}(t,s_t,u_t)=
\left(\begin{array}{c}
\dot{\theta}_t \\
\big(\tau_{1,t} + \tau_{2,t} + k_{g}\sin(\theta_t) - k_{s_0}\theta_t - k_{d_0}\dot{\theta}_t\big)/I
\end{array}\right)
\label{eq:drift_pendulum}
\end{equation}
\noindent where $k_{n}=40\,Nm$, $k_{s}=20\,Nm/rad$ and $k_{d}=2\,Nms/rad$ are constants, $k_{g}=10.754\,Nm$
relates to gravity torque (corresponding to the forearm with a load of
2.268 kg attached at the hand), $k_{d_0}=0.1\,Nm/rad/s$ is a baseline damping factor, $k_{s_0}=1\,Nm/rad$ is a baseline stiffness factor, and
$I=0.337\,kg.m^{2}$ is the moment of inertia of the forearm with the load.
Note that the drift $\tilde{f}$ is nonlinear because of the sine function and the product terms involving the state and control components (bilinearity).
We thus assume that the stochastic state evolves according to the following Ito's SDE:
\begin{equation}
ds_t = \tilde{f}(t, s_t,u_t) \, dt + \tilde{G}(t,s_t,u_t) \, d\tilde{w}_t ,\label{eq:state_variable}
\end{equation}
\noindent with diffusion term
\begin{equation}
\tilde{G}(t,s_t,u_t) = \left(\begin{array}{c}
0 \\
\sigma
\end{array}\right)\label{eq:diffusion_pendulum}
\end{equation}
\noindent where $\tilde{w}$ a scalar Wiener process and $\sigma=0.1$ rad/$s^{2}/\sqrt{s}$ is the motor noise magnitude.

In this model, it is worth noting that $u_{1}-u_{2}$ controls the net joint torque (through the gain $k_n$), whereas $u_{1}+u_{2}$ controls the joint stiffness (through the gain $k_s$) and the joint damping (through the gain $k_d$). Hence the joint's viscoelasticity is proportional to muscle co-contraction, here defined as $u_{1}+u_{2}$.

To model the noisy and delayed sensory integration, we introduce the auxiliary variables $z^p_t$ and $z^v_t$ as follows:
\begin{align}
dz^{p}_{t} &= (\theta_t-z^{p}_{t})/\rho_p \, dt + \gamma_p \, dw^{p}_t \\
dz^{v}_{t} &= (\theta_t-z^{v}_{t})/\rho_v \, dt + \gamma_v \, dw^{v}_t \label{eq:delayed_variable}
\end{align}
\noindent where $w^p$ and $w^v$ denote scalar Wiener processes with noise coefficients $\gamma_p=\gamma_v=0.01$ rad/$\sqrt{s}$ (internal neural noise). Hence $z_{p,t}$ and $z_{v,t}$ are noisy and delayed versions of $\theta_t$ with delay parameters $\rho_p$ and $\rho_v$  (by default set to $0.05\,s$ and $0.1\,s$ respectively).

We then define the augmented state $x_t = (\theta_t, \dot{\theta}_t, z^{p}_{t}, z^{v}_{t})$ and the augmented drift writes as follows:
\begin{equation}
dx_t =\drift(x_t, u_t, t) \, dt + \diff(x_t, u_t, t) \, dw_t ,\label{eq:augmented_variable}
\end{equation}
\noindent where $w=(\tilde{w}, w^p, w^v)$ denotes a standard 3D Wiener process.
The functions $f$ and $G$ are defined as:
\begin{equation}
\drift(t,x_t, u_t) = \left(\begin{array}{c}
\tilde{f}(t,s_t, u_t)\\
(\theta_t - z^{p}_{t})/\rho_p \\
(\theta_t - z^{v}_{t})/\rho_v \\
\end{array}\right) ,\label{eq:aug_f} \quad
\diff(t,x_t, u_t) = \left(\begin{array}{ccc}
0 & 0 & 0\\
\sigma & 0 & 0\\
0 & \gamma_p & 0 \\
0 & 0 & \gamma_v
\end{array}\right).
\end{equation}

As a result, the stochastic system writes as in \eqref{eq:sde}:
\begin{equation}
dx_{t}=\drift(t,x_{t},u_t)\,dt+\diff(t,x_{t},u_t)\,dw_{t},\label{eq:nonlinear-stochastic}
\end{equation}
\noindent with the initial state $x_0 \sim \mathcal{N}(0,\P_0)$ with $\P_0=10^{-6}I_{4\times 4}$.

For the output equation, we assume that we observe the variables $z^{p}_{t}$ and $z^{v}_{t}$ as follows:
\begin{equation}
dy_{t}= \left(\begin{array}{c}
z^{p}_{t} \\
z^{v}_{t}
\end{array}\right) \, dt + \left(\begin{array}{cc}
\eta_p & 0 \\
0 & \eta_v
\end{array}\right) \, dv_{t}.\label{eq:output-stochastic-delayed}
\end{equation}
\noindent where $\eta_p$ and $\eta_v$ are the sensory noise coefficients for proprioception and vision respectively (in rad$\sqrt{s}$). By default, we set to $\eta_p$ = 5 $^\circ\sqrt{s}$ for the proprioceptive noise standard deviation, $\eta_v$ = 1 $^\circ\sqrt{s}$ for the visual noise standard deviation, and $v$ is a standard 2D Wiener process.

This leads to a partial state observation as in \eqref{eq:sde_obs}:
\begin{equation}
dy_{t}= \odrift(t,x_{t}) \, dt + \odiff(t,x_t) \,dv_{t},\label{eq:output-stochastic}
\end{equation}
\noindent where the functions are defined as follows:
\begin{equation}
\odrift(t,x_t)=\left(\begin{array}{c}
z^{p}_{t} \\
z^{v}_{t}
\end{array}\right), \quad \odiff(t,x_t)=\left(\begin{array}{cc}
\eta_p & 0 \\
0 & \eta_v
\end{array}\right).
\end{equation}

To simulate the stabilization of the inverted forearm with minimal squared muscle activations and squared task errors, the following cost function is used:
\begin{equation}
\mathbb{J}(u)=\Exp \left[\int_{0}^{T}\big(u_t^{\top}R u_t+x_{t}^{\top}Q x_{t}\big)\,dt\right]\label{eq:cost_pendulum}
\end{equation}
\noindent where $R=I_{2\times 2}$, an identity matrix, $Q=\textrm{diag}(10,1, 0)$, $Q_\psi=0_{3\times3}$, and $T=5\,s$. Remind that the class of admissible controls are the controls that write $u_t = u^\o(t) + \K(t) \left( \hat{x}_t - m^\o(t)\right), \, t \in [0,T]$, which is the form used in simulations to sample single trajectories.

\paragraph*{Simulation details for the 2-dof reaching task}

The state of the rigid body dynamics is described by joint positions and velocities $(q, \dot{q}) =(q_{1}, q_{2}, \dot{q}_{1},\dot{q}_{2})$ (indices $1$ and $2$ stand for the shoulder and elbow respectively) and 6 muscles are assumed to actuate the arm.

The nonlinear rigid body dynamics of the arm are given by:
\begin{equation} \label{eq:RBD}
    \ddot{q}=\mathcal{M}^{-1}(q)\big(\tau_{m}(q,\dot{q},a) + \tau_{e}(q) -\mathcal{C}(q,\dot{q})\dot{q}\big)
\end{equation}
where $\mathcal{M}$ is the $2\times2$ inertia matrix, $\mathcal{C}$ is the $2\times2$ Coriolis/centripetal matrix, $\tau_m$ is the joint torque produced by muscles, $\tau_e$ is an external torque produced by the environment (here set to $\tau_e=\mathrm{J}^\top \beta \mathrm{x}$ where $\mathrm{J}$ is the Jacobian matrix of the 2-dof arm at the endpoint, $\mathrm{x}$ is the lateral coordinate of the endpoint and $\beta$ is the external force field magnitude in N/m), and $a\,\in\,[0,1]^{6}$ is the vector of muscle activations.

The muscle activation dynamics are modeled as first-order low-pass filters as follows:
\begin{equation} \label{eq:act_dyn}
    \dot{a} = \Gamma^{-1}(u-a)
\end{equation}
where $u \in [0,1]^6$ is the vector of muscle inputs (the control variable), $\Gamma$ is a diagonal matrix whose non-null elements are equal to $\gamma=0.06\,s$, the muscle response times (assumed to be constant here).

The muscle activations generate muscle forces according to the following equation:
\begin{equation}
f_m(l,\dot{l},a) = K(a)\big(l_{r}(a)-l\big)-B(a)\dot{l}
\end{equation}
where $K(a)=\textrm{diag}(k_{0}+k\,a)$, $B(a)=\textrm{diag}(b_{0}+b\,a)$ and $l_{r}(a)=l_{0}+\text{diag}(r_{1},...,r_{6})\,a$ are the activation-dependent muscle stiffness, muscle viscosity and muscle resting length, respectively.

The relationship between muscle forces and muscle torques produced is given using the moment arm matrix, as follows:
\begin{equation}
    \tau_m(q,\dot{q},a)=-M{}^{\top}f_m(l,\dot{l},a)
\end{equation}
where $M$ is the moment arm matrix (assumed to be constant here), $l=l_{m}-Mq$ is the muscle length vector ($l_{m}$ being the muscle length when the joint angle is zero), and $\dot{l}=-M\dot{q}$.

The constant matrix $M$ has the following shape related to the function of each muscle:
\begin{equation}
    M^\top=\left(\begin{array}{cccccc}
    m_{1} & m_{2} & 0 & 0 & m_{5} & m_{6}\\
    0 & 0 & m_{3} & m_{4} & m_{7} & m_{8}
\end{array}\right).
\end{equation}

Muscle parameters are adopted from \cite{Katayama1993} (see their Tables 1, 2, and 3), with the exception of $k, k_0, b,$ and $b_0$, which are multiplied by 10 to simulate stronger muscles capable of resisting the divergent force field.

The dynamics of the neuromechanical system is thus described by a 10-dimensional vector, namely $s=(q, \dot{q}, a)$. To account for sensory information, we assume that delayed joint positions are observed. To this aim, we add the auxiliary variable $z=(z_1, z_2)$ with dynamics
\begin{equation} \label{eq:aux_dyn}
    \dot{z} = \Gamma^{-1}(q-z)
\end{equation}
where $\Gamma = \mathrm{diag}(\rho, \rho)$ with $\rho=0.10\,s$.

Therefore, the augmented state of the system incorporating sensory feedback delays is 12-dimensional and defined as $x_t=(s_t,z_t)=(q_t, \dot{q}_t, a_t, z_t)$ and the dynamics $f(t,x_t,u_t)=(\dot{q}_t, \ddot{q}_t, \dot{a}_t, \dot{z}_t)$ can be written by merging \eqref{eq:RBD}, \eqref{eq:act_dyn}, \eqref{eq:aux_dyn}. For the observation, the function $g$ is simply defined as $g(t,x_t)=z_t$

Next, we model the diffusion terms to describe the effects of sensorimotor noise on the system.
For the motor noise, we assume that a 6-dimensional Wiener process affects the state dynamics through the following diffusion term:
\begin{equation}
    G(t, x_t, u_t)= \left( \begin{array}{cc}
       0_{4\times 6} & 0_{4\times2} \\
       \mathrm{diag}(\sigma_a + \sigma_m \, a_t) & 0_{6\times2} \\
        0_{2\times 6} & \mathrm{diag}(\omega,\omega) \\
    \end{array}
    \right)
\end{equation}
where $\sigma_a=0.025\,\mathrm{a.u.}/\sqrt{s}$, $\sigma_m=0.005\,\mathrm{1}/\sqrt{s}$, $\omega=0.01\,\mathrm{rad}/\sqrt{s}$ are the noise magnitudes. The affine motor noise (including additive and multiplicative components) acts in the space of muscle activations. The noise related to $\omega$ represents internal noise associated with sensory integration.

For the sensory noise, we consider a 2-dimensional Wiener process that affects the observation dynamics through the diffusion term
\begin{equation}
    H(t, x_t)= \left( \begin{array}{cc}
       \eta & 0 \\
       0 & \eta
    \end{array}
    \right),
\end{equation}
where $\eta \in \{0.05, 0.5\}$ $^\circ\sqrt{s}$ (the second value is used to simulate reaching without on-line visual feedback, hence with increased observation noise \cite{Berret2021}).


For this planar reaching task, the cost function is defined as:
\begin{equation}
\mathbb{J}(u) =\mathbb{E}\left[ [x_T-x_r(
T)]^{\top} Q_\phi [x_T-x_r(T)] + \int_{0}^{T} \big( [x_t-x_r(t)]^{\top} Q [x_t-x_r(t)] + u_t^\top R u_t  \big) \, dt\right]
\end{equation}
where $x_r(t)$ for $t\in[0,T]$ implements a minimum hand jerk trajectory \cite{Flash1985} for the joint position and velocities (zeros are appended for $a_t$ and $z_t$ components), $Q = \mathrm{diag}(1\mathrm{e}4, 1\mathrm{e}4, 1\mathrm{e}2, 1\mathrm{e}2, 1_{8\times} )$, $Q_\psi = \mathrm{diag}(1\mathrm{e}4, 1\mathrm{e}4, 1\mathrm{e}2, 1\mathrm{e}2, 0_{8\times} )$ and $R=I_{6\times 6}$. Note that a trajectory tracking task is implemented here for simplicity but the framework does not require a reference trajectory in general (alternatively, a hand jerk cost could be added on the mean trajectory as in \cite{Berret2020b} but this would introduce an additional weighting coefficient).

To define the task completely, the mean initial and final joint positions and velocities are prescribed as hard constraints as in \cite{Berret2020b}. We also set the mean initial muscle activations, which are determined by minimizing the sum of squared activations required to maintain the initial arm posture at zero net torque. The mean initial and final values of $z$ are the same as those of the joint position vector. The initial covariance is set to $\P_0=10^{-6}I_{12\times 12}$. The mean final muscle activations are left free. Movement duration is set to $T=0.75\,s$.

Single trajectories are sampled using the control policy $u_t = u^\o(t) + \K(t) \left( \hat{x}_t - m^\o(t)\right), \, t \in [0,T]$ during forward simulations.

\paragraph*{Numerical resolution}
To numerically solve the approximately equivalent deterministic optimal control problems, we have implemented a direct transcription method using an implicit Euler scheme with a 2-ms time step. The resulting problem has been solved with Julia, using a combination of Symbolics.jl \citep{symbolics.jl}, JuMP.jl \citep{jump.jl} and MadNLP.jl \citep{madnlp.jl2} packages. The codes used to perform the simulations will be made available and attached to the paper.

\section*{Acknowledgments}
This work was supported by the ANR grant ANR-25-NEUC-0001 (NEUROPT) within the CRCNS US-French Research Program. The authors also thank Etienne Burdet for his insightful comments on earlier versions of this manuscript.

\nolinenumbers
\bibliographystyle{naturemag}
\bibliography{BBbib_clean}

@article{FlemingPardoux1982,
	title        = {Optimal control for partially observable diffusions},
	author       = {Wendell H. Fleming and {\'E}tienne Pardoux},
	year         = 1982,
	journal      = {SIAM Journal on Control and Optimization},
	volume       = 20,
	number       = 2,
	pages        = {261--285}
}

@article{Zakai1969,
	title        = {On the optimal filtering of diffusion processes},
	author       = {Moshe Zakai},
	year         = 1969,
	journal      = {Zeitschrift f{\"u}r Wahrscheinlichkeitstheorie und Verwandte Gebiete},
	volume       = 11,
	number       = 3,
	pages        = {230--243}
}

@book{JacobsonMayne1970,
	title        = {Differential Dynamic Programming},
	author       = {David H. Jacobson and David Q. Mayne},
	year         = 1970,
	publisher    = {Elsevier},
	address      = {New York}
}

@inproceedings{Platt2010,
	title        = {Belief space planning assuming maximum likelihood observations},
	author       = {R. Platt AND R. Tedrake AND L. Kaelbling AND T. Lozano-Perez},
	year         = 2010,
	month        = {June},
	booktitle    = {Proceedings of Robotics: Science and Systems},
	address      = {Zaragoza, Spain}
}

@article{VandenBerg2012,
	title        = {Motion planning under uncertainty using iterative local optimization in belief space},
	author       = {Jur van~den~Berg and Sachin Patil and Ron Alterovitz},
	year         = 2012,
	journal      = {Int. J. Robotics Res.},
	volume       = 31,
	number       = 11,
	pages        = {1263--1278}
}

@inproceedings{FanMing2021,
	title        = {Model-based reinforcement learning for continuous control with posterior sampling},
	author       = {Yifan Fan and Yi Ming},
	year         = 2021,
	booktitle    = {Proceedings of the 38th International Conference on Machine Learning (ICML)},
	publisher    = {PMLR},
	pages        = {3078--3087}
}

@inproceedings{Wang2023,
	title        = {Learning belief representations for partially observable deep reinforcement learning},
	author       = {Annie Wang and Alex C. Li and Thomas Q. Klassen and Rodrigo Toro Icarte and Sheila A. McIlraith},
	year         = 2023,
	booktitle    = {Proceedings of the Conference on Machine Learning Research (PMLR)},
	publisher    = {PMLR}
}

@inproceedings{Okada2020,
  title = {PlaNet of the Bayesians: Reconsidering and Improving Deep Planning Network by Incorporating Bayesian Inference},
  booktitle = {2020 IEEE/RSJ International Conference on Intelligent Robots and Systems (IROS)},
  publisher = {IEEE},
  author = {Okada,  Masashi and Kosaka,  Norio and Taniguchi,  Tadahiro},
  year = {2020},
  month = oct,
  pages = {5611–5618}
}

@article{Mesbah2018,
	title        = {Stochastic model predictive control with active uncertainty learning: A Survey on dual control},
	author       = {Ali Mesbah},
	year         = 2018,
	journal      = {Annual Reviews in Control},
	volume       = 45,
	pages        = {107--117},
	issn         = {1367-5788}
}

@book{betts2020,
	title        = {Practical methods for optimal control using nonlinear programming},
	author       = {Betts, John T.},
	year         = 2020,
	booktitle    = {Practical methods for optimal control using nonlinear programming},
	publisher    = {Society for Industrial and Applied Mathematics},
	address      = {Philadelphia},
	series       = {Advances in design and control ; 36},
	isbn         = 9781611976199,
	edition      = {Third edition.},
	language     = {eng},
	lccn         = 2020004553
}

@article{Williams2017,
	title        = {Model Predictive Path Integral Control: From Theory to Parallel Computation},
	author       = {Williams,  Grady and Aldrich,  Andrew and Theodorou,  Evangelos A.},
	year         = 2017,
	month        = feb,
	journal      = {Journal of Guidance,  Control,  and Dynamics},
	publisher    = {American Institute of Aeronautics and Astronautics (AIAA)},
	volume       = 40,
	number       = 2,
	pages        = {344–357},
	issn         = {1533-3884}
}

@article{Ernst2004,
	title        = {Merging the senses into a robust percept},
	author       = {Ernst,  Marc O. and B\"{u}lthoff,  Heinrich H.},
	year         = 2004,
	month        = apr,
	journal      = {Trends in Cognitive Sciences},
	publisher    = {Elsevier BV},
	volume       = 8,
	number       = 4,
	pages        = {162–169},
	issn         = {1364-6613}
}

@article{Voloshina2013,
	title        = {Biomechanics and energetics of walking on uneven terrain},
	author       = {Voloshina,  Alexandra S. and Kuo,  Arthur D. and Daley,  Monica A and Ferris,  Daniel P.},
	year         = 2013,
	month        = jan,
	journal      = {Journal of Experimental Biology},
	publisher    = {The Company of Biologists},
	issn         = {0022-0949}
}

@article{jump.jl,
	title        = {{JuMP} 1.0: {R}ecent improvements to a modeling language for mathematical optimization},
	author       = {Miles Lubin and Oscar Dowson and Joaquim {Dias Garcia} and Joey Huchette and Beno{\^i}t Legat and Juan Pablo Vielma},
	year         = 2023,
	journal      = {Mathematical Programming Computation},
	volume       = 15,
	pages        = {581–589}
}

@article{symbolics.jl,
	title        = {High-Performance Symbolic-Numerics via Multiple Dispatch},
	author       = {Gowda, Shashi and Ma, Yingbo and Cheli, Alessandro and Gw\'{o}\'{z}zd\'{z}, Maja and Shah, Viral B. and Edelman, Alan and Rackauckas, Christopher},
	year         = 2022,
	month        = {jan},
	journal      = {ACM Commun. Comput. Algebra},
	publisher    = {Association for Computing Machinery},
	address      = {New York, NY, USA},
	volume       = 55,
	number       = 3,
	pages        = {92–96},
	issn         = {1932-2240},
	issue_date   = {September 2021},
	numpages     = 5
}

@article{VanWouwe2022,
	title        = {An approximate stochastic optimal control framework to simulate nonlinear neuro-musculoskeletal models in the presence of noise},
	author       = {Van Wouwe, Tom and Ting, Lena H. and De Groote, Friedl},
	year         = 2022,
	month        = jun,
	journal      = {PLOS Computational Biology},
	publisher    = {Public Library of Science (PLoS)},
	volume       = 18,
	number       = 6,
	pages        = {e1009338},
	issn         = {1553-7358},
	editor       = {Haith, Adrian M.}
}

@article{Berret2024,
	title        = {Co-contraction embodies uncertainty: An optimal feedforward strategy for robust motor control},
	author       = {Berret,  Bastien and Verdel,  Dorian and Burdet,  Etienne and Jean,  Fr{\'e}d{\'e}ric},
	year         = 2024,
	month        = nov,
	journal      = {PLOS Computational Biology},
	publisher    = {Public Library of Science (PLoS)},
	volume       = 20,
	number       = 11,
	pages        = {e1012598},
	issn         = {1553-7358},
	editor       = {Haith,  Adrian M}
}

@article{Berret2020b,
	title        = {Stochastic optimal open-loop control as a theory of force and impedance planning via muscle co-contraction},
	author       = {Bastien Berret and Fr{\'{e}}d{\'{e}}ric Jean},
	year         = 2020,
	month        = feb,
	journal      = {{PLOS} Computational Biology},
	publisher    = {Public Library of Science ({PLoS})},
	volume       = 16,
	number       = 2,
	pages        = {e1007414},
	editor       = {Adrian M Haith}
}

@article{Berret2021,
	title        = {Stochastic optimal feedforward-feedback control determines timing and variability of arm movements with or without vision},
	author       = {Bastien Berret and Adrien Conessa and Nicolas Schweighofer and Etienne Burdet},
	year         = 2021,
	month        = jun,
	journal      = {{PLOS} Computational Biology},
	publisher    = {Public Library of Science ({PLoS})},
	volume       = 17,
	number       = 6,
	pages        = {e1009047},
	editor       = {Ulrik R. Beierholm}
}

@article{Athans1971,
	title        = {The Role and Use of the Stochastic Linear-Quadratic-Gaussian Problem in Control System Design},
	author       = {M. Athans},
	year         = 1971,
	journal      = {IEEE Trans. Autom. Control},
	volume       = 16,
	number       = 6,
	pages        = {529--552}
}

@article{Berret2011b,
	title        = {Evidence for composite cost functions in arm movement planning: an inverse optimal control approach.},
	author       = {Berret, Bastien and Chiovetto, Enrico and Nori, Francesco and Pozzo, Thierry},
	year         = 2011,
	month        = oct,
	journal      = {PLOS Computational Biology},
	volume       = 7,
	number       = 10,
	pages        = {e1002183},
	language     = {eng},
	medline-pst  = {ppublish},
	pii          = {PCOMPBIOL-D-11-00355},
	pmid         = 22022242,
	school       = {Italian Institute of Technology, Department of Robotics, Brain and Cognitive Sciences, Genoa, Italy. bastien.berret\@iit.it}
}

@article{Berret2016,
	title        = {Why Don't We Move Slower? The Value of Time in the Neural Control of Action.},
	author       = {Berret, Bastien and Jean, Fr{\'{e}}d{\'{e}}ric},
	year         = 2016,
	month        = jan,
	journal      = {J. Neurosci.},
	volume       = 36,
	number       = 4,
	pages        = {1056--1070},
	language     = {eng},
	medline-pst  = {ppublish},
	pii          = {36/4/1056},
	pmid         = 26818497,
	school       = {Unit� de Math�matiques Appliqu�es (UMA), Ecole Nationale Sup�rieure de Techniques Avanc�es, ParisTech, Universit� Paris-Saclay, F-91120 Palaiseau, France.}
}

@article{Burdet2001,
	title        = {The central nervous system stabilizes unstable dynamics by learning optimal impedance.},
	author       = {E. Burdet and R. Osu and D. W. Franklin and T. E. Milner and M. Kawato},
	year         = 2001,
	month        = nov,
	journal      = {Nature},
	volume       = 414,
	number       = 6862,
	pages        = {446--449},
	language     = {eng},
	medline-pst  = {ppublish},
	pii          = 35106566,
	pmid         = 11719805,
	school       = {Department of Mechanical Engineering, National University of Singapore, 119260, Singapore.}
}

@article{Flash1985,
	title        = {The coordination of arm movements: an experimentally confirmed mathematical model.},
	author       = {T. Flash and N. Hogan},
	year         = 1985,
	month        = jul,
	journal      = {J. Neurosci.},
	volume       = 5,
	number       = 7,
	pages        = {1688--1703},
	pmid         = 4020415
}

@article{Harris1998,
	title        = {Signal-dependent noise determines motor planning.},
	author       = {C. M. Harris and D. M. Wolpert},
	year         = 1998,
	month        = aug,
	journal      = {Nature},
	volume       = 394,
	number       = 6695,
	pages        = {780--784},
	pmid         = 9723616
}

@article{Leib2023,
	title        = {Behavioral Motor Performance},
	author       = {Leib,  Raz and Howard,  Ian S. and Millard,  Matthew and Franklin,  David W.},
	year         = 2023,
	month        = dec,
	journal      = {Comprehensive Physiology},
	publisher    = {Wiley},
	volume       = 14,
	number       = 1,
	pages        = {5179–5224},
	issn         = {2040-4603}
}

@article{Li2007,
	title        = {Iterative linearization methods for approximately optimal control and estimation of non-linear stochastic system},
	author       = {W. Li and E. Todorov},
	year         = 2007,
	journal      = {Int. J. Control},
	volume       = 80,
	number       = 9,
	pages        = {1439--1453}
}

@book{stengel1986optimal,
	title        = {Optimal Control and Estimation},
	author       = {Stengel, R.F.},
	year         = 1986,
	publisher    = {Dover Publications},
	series       = {Dover books on advanced mathematics},
	isbn         = 9780486682006,
	lccn         = {lc94020406}
}

@article{Todorov2009,
	title        = {Efficient computation of optimal actions.},
	author       = {E. Todorov},
	year         = 2009,
	month        = jul,
	journal      = {Proc. Natl. Acad. Sci. U.S.A.},
	volume       = 106,
	number       = 28,
	pages        = {11478--11483},
	language     = {eng},
	medline-pst  = {ppublish},
	pii          = {0710743106},
	pmid         = 19574462,
	school       = {Engineering, University of Washington, Seattle, WA 98195, USA.}
}

@article{Todorov2002,
	title        = {Optimal feedback control as a theory of motor coordination.},
	author       = {E. Todorov and M. I. Jordan},
	year         = 2002,
	month        = nov,
	journal      = {Nat. Neurosci.},
	volume       = 5,
	number       = 11,
	pages        = {1226--1235},
	pii          = {nn963},
	pmid         = 12404008
}

@inproceedings{Todorov2005b,
	title        = {A generalized iterative LQG method for locally-optimal feedback control of constrained nonlinear stochastic systems},
	author       = {Todorov, E. and Weiwei Li},
	year         = 2005,
	month        = jun,
	booktitle    = {Proceedings of the 2005 American Control Conference},
	pages        = {300--306 vol. 1},
	issn         = {0743-1619}
}

@article{Wolpert1995c,
	title        = {An internal model for sensorimotor integration.},
	author       = {D. M. Wolpert and Z. Ghahramani and M. I. Jordan},
	year         = 1995,
	month        = sep,
	journal      = {Science},
	volume       = 269,
	number       = 5232,
	pages        = {1880--1882},
	pmid         = 7569931
}

@article{Hogan1984a,
	title        = {Adaptive control of mechanical impedance by coactivation of antagonist muscles},
	author       = {Hogan, Neville},
	year         = 1984,
	journal      = {IEEE Trans. Autom. Control},
	publisher    = {IEEE},
	volume       = 29,
	number       = 8,
	pages        = {681--690}
}

@article{Katayama1993,
	title        = {Virtual trajectory and stiffness ellipse during multijoint arm movement predicted by neural inverse models.},
	author       = {Katayama, M and Kawato, M},
	year         = 1993,
	journal      = {Biol. Cybern.},
	volume       = 69,
	pages        = {353--362},
	issn         = {0340-1200},
	citation-subset = {IM, S},
	completed    = {1994-02-08},
	country      = {Germany},
	issn-linking = {0340-1200},
	issue        = {5-6},
	nlm-id       = 7502533,
	pmid         = 8274536,
	pubmodel     = {Print},
	pubstatus    = {ppublish},
	revised      = {2017-08-26}
}

@book{maybeck1982stochastic,
	title        = {Stochastic models, estimation, and control},
	author       = {Maybeck, Peter S},
	year         = 1982,
	publisher    = {Academic press},
	volume       = 3
}

@article{Franklin2003,
	title        = {Adaptation to stable and unstable dynamics achieved by combined impedance control and inverse dynamics model.},
	author       = {Franklin, David W and Osu, Rieko and Burdet, Etienne and Kawato, Mitsuo and Milner, Theodore E},
	year         = 2003,
	month        = nov,
	journal      = {J. Neurophysiol.},
	volume       = 90,
	pages        = {3270--3282},
	issn         = {0022-3077},
	citation-subset = {IM},
	completed    = {2004-01-22},
	country      = {United States},
	issn-linking = {0022-3077},
	issue        = 5,
	nlm-id       = {0375404},
	pii          = {90/5/3270},
	pmid         = 14615432,
	pubmodel     = {Print},
	pubstatus    = {ppublish},
	revised      = {2006-11-15}
}

@article{Scott2012,
	title        = {The computational and neural basis of voluntary motor control and planning.},
	author       = {Scott, Stephen H},
	year         = 2012,
	month        = nov,
	journal      = {Trends in cognitive sciences},
	volume       = 16,
	pages        = {541--549},
	issn         = {1879-307X},
	citation-subset = {IM},
	completed    = {2013-03-25},
	country      = {England},
	issn-linking = {1364-6613},
	issue        = 11,
	nlm-id       = 9708669,
	pii          = {S1364-6613(12)00224-0},
	pmid         = 23031541,
	pubmodel     = {Print-Electronic},
	pubstatus    = {ppublish},
	revised      = {2012-10-26}
}

@inproceedings{Vanderborght2012,
	title        = {Variable impedance actuators: Moving the robots of tomorrow},
	author       = {B. Vanderborght and A. Albu-Schaeffer and A. Bicchi and E. Burdet and D. Caldwell and R. Carloni and M. Catalano and G. Ganesh and M. Garabini and M. Grebenstein and G. Grioli and S. Haddadin and A. Jafari and M. Laffranchi and D. Lefeber and F. Petit and S. Stramigioli and N. Tsagarakis and M. Van Damme and R. Van Ham and L. C. Visser and S. Wolf},
	year         = 2012,
	month        = oct,
	booktitle    = {Proc. IEEE/RSJ Int. Conf. Intelligent Robots and Systems},
	pages        = {5454--5455},
	issn         = {2153-0866}
}

@article{Franklin2011,
	title        = {Computational mechanisms of sensorimotor control},
	author       = {Franklin, David W and Wolpert, Daniel M},
	year         = 2011,
	journal      = {Neuron},
	publisher    = {Elsevier},
	volume       = 72,
	number       = 3,
	pages        = {425--442}
}

@misc{Franklin2023,
	title        = {Visuomotor feedback tuning in the absence of visual error information},
	author       = {Franklin,  Sae and Franklin,  David W.},
	year         = 2023,
	publisher    = {arXiv},
	copyright    = {Creative Commons Attribution Non Commercial No Derivatives 4.0 International}
}

@article{Franklin2007,
	title        = {Visual feedback is not necessary for the learning of novel dynamics},
	author       = {Franklin, David W and So, Udell and Burdet, Etienne and Kawato, Mitsuo},
	year         = 2007,
	journal      = {PLoS One},
	publisher    = {Public Library of Science},
	volume       = 2,
	number       = 12,
	pages        = {e1336}
}

@article{Franklin2004,
	title        = {Impedance control balances stability with metabolically costly muscle activation},
	author       = {Franklin, David W and So, Udell and Kawato, Mitsuo and Milner, Theodore E},
	year         = 2004,
	journal      = {J. Neurophysiol.},
	publisher    = {American Physiological Society},
	volume       = 92,
	number       = 5,
	pages        = {3097--3105}
}

@book{Bryson1969,
	title        = {Applied Optimal Control},
	author       = {Bryson, A. E. and Ho, Y. C.},
	year         = 1969,
	booktitle    = {Applied Optimal Control},
	publisher    = {Blaisdell},
	address      = {New York},
	description  = {CCNLab BibTeX}
}

@article{Phan2020,
	title        = {Estimating Human Wrist Stiffness during a Tooling Task},
	author       = {Phan,  Gia-Hoang and Hansen,  Clint and Tommasino,  Paolo and Budhota,  Aamani and Mohan,  Dhanya Menoth and Hussain,  Asif and Burdet,  Etienne and Campolo,  Domenico},
	year         = 2020,
	month        = jun,
	journal      = {Sensors},
	publisher    = {MDPI AG},
	volume       = 20,
	number       = 11,
	pages        = 3260,
	issn         = {1424-8220}
}

@article{Theodorou2010,
	title        = {A generalized path integral control approach to reinforcement learning},
	author       = {Theodorou, Evangelos and Buchli, Jonas and Schaal, Stefan},
	year         = 2010,
	month        = dec,
	journal      = {J. Mach. Learn. Res.},
	publisher    = {JMLR.org},
	volume       = 11,
	pages        = {3137--3181}
}

@inproceedings{Theodorou2010a,
	title        = {Stochastic differential dynamic programming},
	author       = {Theodorou, Evangelos and Tassa, Yuval and Todorov, Emo},
	year         = 2010,
	booktitle    = {Proceedings of the 2010 American Control Conference},
	pages        = {1125--1132},
	organization = {IEEE}
}

@article{Berret2020,
	title        = {Efficient computation of optimal open-loop controls for stochastic systems},
	author       = {Bastien Berret and Fr\'{e}d\'{e}ric Jean},
	year         = 2020,
	journal      = {Automatica},
	volume       = 115,
	pages        = 108874,
	issn         = {0005-1098}
}

@article{Leparoux2024,
	title        = {Statistical linearization for robust motion planning},
	author       = {Clara Leparoux and Riccardo Bonalli and Bruno H\'{e}riss\'{e} and Fr\'{e}d\'{e}ric Jean},
	year         = 2024,
	journal      = {Systems \& Control Letters},
	volume       = 189,
	pages        = 105825,
	issn         = {0167-6911}
}

@book{Davis1977,
	title        = {Linear Estimation and Stochastic Control},
	author       = {Davis, M.H.A.},
	year         = 1977,
	publisher    = {Chapman and Hall},
	series       = {A Halsted Press Book},
	isbn         = 9780412154706,
	lccn         = 77023389
}

@book{bensoussan1992stochastic,
	title        = {Stochastic Control of Partially Observable Systems},
	author       = {Bensoussan,  Alain},
	year         = 1992,
	month        = aug,
	publisher    = {Cambridge University Press},
	isbn         = 9780511526503
}

@book{Burdet2013,
	title        = {Human Robotics: Neuromechanics and Motor Control},
	author       = {Burdet,  Etienne and Franklin,  David W. and Milner,  Theodore E.},
	year         = 2013,
	month        = sep,
	publisher    = {The MIT Press},
	isbn         = 9780262314817
}

@article{madnlp.jl2,
	title        = {Accelerating optimal power flow with {GPU}s: {SIMD} abstraction of nonlinear programs and condensed-space interior-point methods},
	author       = {Shin, Sungho and Anitescu, Mihai and Pacaud, Fran{\c{c}}ois},
	year         = 2024,
	journal      = {Electric Power Systems Research},
	publisher    = {Elsevier},
	volume       = 236,
	pages        = 110651
}

\end{document}